\documentclass[a4paper, 12pt]{article}
\newcommand{\toukou}[1]{\ifx\TOUKOU\undefined\else{#1}\fi}%
\newcommand{\toukoudel}[1]{\ifx\TOUKOU\undefined{#1}\else\fi}%
\newcommand{\toukouchange}[2]{\ifx\TOUKOU\undefined{#1}\else{#2}\fi}%
\usepackage{amssymb}
\usepackage{amsmath}
\usepackage{amsthm}
\usepackage[dvipdfmx]{graphicx}
\usepackage{verbatim,enumerate}
\usepackage[active]{srcltx}
 \newtheorem{theorem}{Theorem}[section]
 \newtheorem*{theorem*}{Theorem}
 \newtheorem*{lemma*}{Lemma}
 \newtheorem{proposition}[theorem]{Proposition}
 
 \newtheorem{fact*}{Fact}
 \newtheorem{lemma}[theorem]{Lemma}
 
\theoremstyle{definition}
 \newtheorem{definition}[theorem]{Definition}
 \newtheorem{remark}[theorem]{Remark}
 \newtheorem*{remark*}{Remark}
 \newtheorem*{acknowledgements}{Acknowledgements}
 
\numberwithin{equation}{section}

\newcommand{\R}{\boldsymbol{R}}

\newcommand{\hess}{\operatorname{Hess}}

\renewcommand{\phi}{\varphi}

\newcommand{\inner}[2]{\left\langle{#1},{#2}\right\rangle}

\newcommand{\A}{\mathcal{A}}

\newcommand{\F}{\mathcal{F}}
\newcommand{\M}{\mathcal{M}}
\newcommand{\C}{\mathcal{C}}
\newcommand{\N}{\mathcal{N}}

\newcommand{\pmt}[1]{{\begin{pmatrix} #1  \end{pmatrix}}}

\newcommand{\mycomment}[1]{}

\usepackage{array}
\newcolumntype{L}{>{\displaystyle}l}
\newcolumntype{C}{>{\displaystyle}c}
\newcolumntype{R}{>{\displaystyle}r}
\toukouchange{
\makeatletter
  \newcommand{\subsubsubsection}{\@startsection{paragraph}{4}{\z@}%
    {-1ex \@plus -1ex \@minus -.2ex}%
    {1.0ex \@plus.2ex}
    {\reset@font\bfseries\normalsize}
  }
  \makeatother
  \setcounter{secnumdepth}{4}
}
{

\articleinfo{}{}{}

\setcounter{page}{1}
\usepackage[mathscr]{eucal}

\title[geometric foliations on a cuspidal edge]{On pairs of geometric foliations on a cuspidal edge}
\author[Kentaro Saji]{Kentaro Saji}
\address{Department of Mathematics, Graduate School of Science, 
Kobe University,\\ Rokko, Nada, Kobe 657-8501, Japan
}
\email{sajiO\!\!\!amath.kobe-u.ac.jp}
\rcvdate{}
\rvsdate{}
}
\begin{document}
\toukouchange{
\begin{center}
{\large {\bf On pairs of geometric foliations on a cuspidal edge}}
\\[2mm]
\today
\\[2mm]
Kentaro Saji

\begin{quote}
{\small 
We study the  topological configurations of the lines of
principal curvature, the asymptotic and characteristic curves on a cuspidal
edge, in the domain of a parametrization of this surface as well as
on the surface itself. Such configurations
are determined by the $3$-jets of a parametrization of the surface.
}
\end{quote}
\end{center}
}{
\begin{abstract}
We study the  topological configurations of the lines of
principal curvature, the asymptotic and characteristic curves on a cuspidal
edge, in the domain of a parametrization of this surface as well as
on the surface itself. Such configurations
are determined by the $3$-jets of a parametrization of the surface.
\end{abstract}
\maketitle
}
\renewcommand{\thefootnote}{\fnsymbol{footnote}}
\footnote[0]{Partly supported by the
Japan Society for the Promotion of Science (JSPS)
and the
Coordenadoria de Aperfei\c{c}oamento de Pessoal de N\'ivel Superior
under the Japan-Brazil research cooperative
program and the
Grant-in-Aid for
Scientific Research (Young Scientists (B)) No. 26400087, from JSPS.}
\\[4mm]
\section{Introduction and preliminaries about cuspidal edges}
A singular point $x$ of a map $f:(\R^2,x)\to(\R^3,0)$ 
is called a {\it cuspidal edge\/}
if the map-germ $f$ at $x$ is $\mathcal{A}$-equivalent to
$(u,v)\mapsto(u,v^2,v^3)$ at $0$. (Two map-germs
$f_1,f_2:(\R^n,0)\to(\R^m,0)$ are $\mathcal{A}$-{\it
equivalent}\/ if there exist diffeomorphisms
$S:(\R^n,0)\to(\R^n,0)$ and $T:(\R^m,0)\to(\R^m,0)$ such
that $ f_2\circ S=T\circ f_1 $.) 
If the singular point $x$
of $f$ is a cuspidal edge, then $f$ at $x$ is a front
in the sense of \cite{AGV} (see also \cite{krsuy}), and
furthermore, they are one of two types of generic singularities of
fronts (the other one is a {\it swallowtail}\/ which is a singular
point $u$ of $f$ satisfying that $f$ at $u$ is
$\mathcal{A}$-equivalent to $(u,v)\mapsto(u, u^2v + 3u^4, 2 uv + 4
u^3)$ at $0$). 

It is shown in \cite{MS} that a cuspidal edge can locally be
parametrized after smooth changes of coordinates in the source and
isometries in the target by
\begin{equation}\label{eq:west1}
f(u,v)
=\left(u,a_1(u)+\dfrac{v^2}{2},b_2(u)+v^2b_3(u)+v^3b_4(u,v)\right),
\end{equation}
where 
$a_1(u),b_2(u),b_3(u),b_4(u,v)$ are $C^\infty$-functions
satisfying
$a_1(0)=a_1'(0)=b_2(0)=b_2'(0)=b_3(0)=0$, 
$b_2''(0)>0$, $b_4(0,0)\ne0$.
Writing 
$a_1(u)=a_{20}u^2/2+a_{30}u^3/6+u^4h_1(u)$,
$b_2(u)=b_{20}u^2/2+b_{30}u^3/6+u^4h_2(u)$,
$b_3(u)=b_{12}u/2+u^2h_3(u)$,
$b_4(u,v)=b_{03}/6+uh_4(u)+vh_5(u,v)$,
we have
\begin{equation}
\label{eq:west2}
\begin{array}{L}
f(u,v) 
= \displaystyle \Big(u,
\frac{a_{20}}{2}u^2+\frac{a_{30}}{6}u^3+\frac{v^2}{2},\\
\hspace{20mm}
\frac{b_{20}}{2}u^2+\frac{b_{30}}{6}u^3+\frac{b_{12}}{2}uv^2
+\frac{b_{03}}{6}v^3\Big) + h(u,v),
\end{array}
\end{equation}
where $b_{03}\ne0,\ b_{20}\geq0$ and
$$h(u,v)= \big(
0,\, u^4h_1(u), u^4h_2(u)+u^2v^2h_3(u)+uv^3h_4(u)+v^4h_5(u,v) \big),
$$
with $h_1(u),h_2(u),h_3(u),h_4(u),h_5(u,v)$  smooth functions.
Several differential geometric invariants of cuspidal edges
are investigated (\cite{ist,MS,MSUY,nuy,ot,front}), 
and coefficients of
\eqref{eq:west2} are such invariants.
According to \cite{MS}, it is
known that $a_{20}$ coincides with the singular curvature
$\kappa_s$, $b_{20}$ coincides with the limiting normal curvature
$\kappa_\nu$, $b_{03}$ coincides with the cuspidal curvature 
$\kappa_c$ and $b_{12}$ coincides with the cusp-directional torsion
$\kappa_t$ at the origin.
The singular curvature is the geodesic curvature
of the singular set with sign,
and the limiting normal curvature is the normal curvature
of the singular set, and they relates to
the shape of cuspidal edge
(see \cite{front}).
The cuspidal curvature measures the wideness
of the cusp, and
the cusp-directional torsion measures
the rotating ratio of the cusp along
the singular set
(see \cite{MS,MSUY}).

On the other hand,
let $U\subset \R^2$ be an open subset and $(u,v)$ a
coordinate system on $U$.
Let 
\begin{equation}\label{eq:BDE}
\omega
=
a(u,v) dv^2 + 2b(u,v) dudv + c(u,v) du^2
\end{equation}
be a $2$-tensor on $U$,
where $a,b,c$ are smooth functions, called the {\em coefficients\/}
of $\omega$. 
We call $\omega=0$ 
a {\em binary differential equation\/ $($BDE\/$)$ 
corresponding to\/ $\omega$}.
If $b^2-ac>0$ at $x\in U$, then
$\omega(x)=0$ defines two directions 
in $T_xU$,
and
integral curves of these directions
for two smooth and transverse foliations,
called {\it foliations with respect to\/} $\omega$.
If $b^2-ac=0$ at $x\in U$, then
generically $\omega(x)=0$ defines a single direction,
and the integral curves form in general a family of cusps.
Thus we are mainly interested in behavior of 
integral curves of a BDE near a point
where $b^2-ac$ vanishes.
We call {\it discriminant} of a BDE the set where $b^2-ac=0$.
If the single direction is transverse to 
the discriminant, then 
the BDE is equivalent to $dv^2 +udu^2 = 0$ (\cite{c,dara}).
The normal form for the stable cases
when the single direction is tangent to the discriminant
is obtained in \cite{d1,d2}.
Topological classifications of generic families
of BDEs are obtained in 
\cite{BFT,BTbinary,BTimplicit,joey,faridtari,t}.
On the other hand,
BDEs as geometric foliations on surfaces in three space is
studied in \cite{BF,ggs,guinez1,guinez2,gg,faridtari}.
See \cite{diis,hiiy,if} for other approaches for
geometric foliations.

In this paper, following
\cite{ggs,gs,guinez1,joey,faridtari}, we stick 
to special BDEs from differential geometry
of surface in $\R^3$.
There are three fundamental BDEs
on a regular surface in $\R^3$.
Let $f:(\R^2,0)\to(\R^3,0)$ be a regular surface
with a unit normal vector field $\nu$.
Let 
$\omega_{lc}$,
$\omega_{as}$ and 
$\omega_{ch}$
be $2$-tensors defined by
$$
\begin{array}{rcl}
\omega_{lc}&=&(EM-FL)dv^2+(EN-GL)dudv+(FN - GM)du^2,\\
\omega_{as}&=&N\,dv^2+2M\,dudv+L\,du^2,\\
\omega_{ch}&=&\big(2M(GM-FN) - N(GL-EN)\big)\,dv^2 \\
&&\hspace{10mm}+2\big(M(GL+EN) - 2FLN\big)\,dudv \\
&&\hspace{20mm}+\big(L(GL-EN) -2M(FL-EM)\big)\,du^2,
\end{array}
$$
where $(u,v)$ is a coordinate system on $(\R^2,0)$, and
$$
\pmt{E&F\\F&G}=
\pmt{\inner{f_u}{f_u}&\inner{f_u}{f_v}\\
\inner{f_v}{f_u}&\inner{f_v}{f_v}},\ 
\pmt{L&M\\M&N}=
\pmt{\inner{f_{uu}}{\nu}&\inner{f_{uv}}{\nu}\\
\inner{f_{uv}}{\nu}&\inner{f_{vv}}{\nu}},
$$
where $\inner{~}{~}$ stands for the standard
inner product of $\R^3$.
Each integral curve of the foliations with respect to
$\omega_{lc}$ is called 
a {\it line of curvature\/},
each integral curve with respect to
$\omega_{as}$ is called an
{\it asymptotic curve\/}
and
each integral curve with respect to
$\omega_{ch}$ is called a {\it characteristic curve}\/
or {\it harmonic mean curvature curve}.
Asymptotic curves appear only on 
a domain where the Gaussian curvature $K$ of $f$ is non-negative,
and 
characteristic curves appear only on 
a domain where the Gaussian curvature $K$ of $f$ is non-positive.
Since $\omega_{ch}=0$ can be deformed as
$$
\begin{array}{cl}
&(NH-GK)\,dv^2+2(MH-FK)\,dudv+(LH-EK)\,du^2=0\\[2mm]
\Leftrightarrow&
\dfrac{N\,dv^2+2M\,dudv+L\,du^2}
{G\,dv^2+2F\,dudv+E\,du^2}=\dfrac{K}{H}
\left(=
\dfrac{2}{\kappa_1^{-1}+\kappa_2^{-1}}\right),
\end{array}
$$
where $K$ is the Gaussian curvature, $H$ is
the mean curvature, and $\kappa_1$, $\kappa_2$ are
the principal curvatures of $f$,
we see that
along the characteristic curve, the
normal curvature of it is equal to the harmonic
mean of the principal curvatures
(see \cite{gs}, for example).

\begin{acknowledgements}
This paper is prepared while the author was visiting
Luciana Martins 
at IBILCE - UNESP.
He would like to thank Luciana Martins for fruitful
discussions.
He would also like to thank Farid Tari for
valuable comments.
He would also like to thank the
referee for careful reading and helpful suggestions.
\end{acknowledgements}
\section{Preliminaries on BDEs}
In this section, following  \cite{BTbinary,BTimplicit},
we introduce a method to study 
the configurations of the solution curves of
a BDE.
Let $\omega(u,v)$ be 
the $2$-tensor on $(U;(u,v))\subset \R^2$
as in \eqref{eq:BDE}.
If $(a,b,c)\ne(0,0,0)$ at $x\in U$, then 
$\omega$ is called of {\em Type\/ $1$} at $x$,
and if $(a,b,c)=(0,0,0)$ at $x\in U$, then 
$\omega$ is called of  {\em Type\/ $2$} at $x$.
If $\omega$ is of Type 2 at $x$,
then $\delta=b^2-ac$ has a critical point at $x$.
Since we are interested in local behavior of $\omega$,
we set $x=(0,0)$.
If $\omega$ is of 
Type 1, then $\omega$ defines a single direction at points on
$\Delta$, and if it is
of Type 2, then all directions
in the plane are solutions of $\omega=0$ at that point.
Moreover, if $\omega$ is of Type $2$ at $x$, then
$\Delta$ is not a smooth curve.
We are interested in 
the configurations of the foliations 
of $\omega=0$.
We define the following equivalence.
\begin{definition}
Two binary differential equations 
$\omega_1=0$ and $\omega_2=0$
are {\em equivalent\/} if
there exist a diffeomorphism germ
$\Phi:(\R^2,0)\to(\R^2,0)$
and a non-zero function $\rho:(\R^2,0)\to\R$ such that
$\rho\big(\Phi^*\omega_1\big)=\omega_2$
holds.
If $\Phi$ is a homeomorphism
such that $\Phi$ takes the integral curves of $\omega_1$ 
to those of $\omega_2$,
they are called {\em topologically equivalent}.
\end{definition}
If two binary differential equations are equivalent then 
the configurations of their
foliations can be regarded the same.
To obtain the topological configurations,
we use the following method developed in 
\cite{BTbinary,BTimplicit,ggs,guinez1,faridtari,faridsurvey}.
We separate our consideration into the following three cases:
\begin{itemize}
\item Case 1: $(a(0),b(0),c(0))\ne(0,0,0)$ and $\delta(0)\ne0$ 
(Type 1).
\item Case 2: $(a(0),b(0),c(0))\ne(0,0,0)$ and $\delta(0)=0$ (Type 1).
\item Case 3: $(a(0),b(0),c(0))=(0,0,0)$ (Type 2).
\end{itemize}
Consider the associated surface to $\omega$
$$
\M=\{(u,v,[\alpha,\beta])\in(\R^2,0)\times \R P^1\,|\,
a\beta^2+2b\alpha\beta+c\alpha^2=0\}.
$$
Then $\M$ is a smooth manifold if $b^2-ac\ne0$, 
or if $a_u+2b_up+c_up^2=0$ and
$a_v+2b_vp+c_vp^2=0$ do not have any
common root.
The second condition is equivalent to
\begin{equation}\label{eq:mmfd}
\det\pmt{
a_u&2b_u&c_u&0\\
0&a_u&2b_u&c_u\\
a_v&2b_v&c_v&0\\
0&a_v&2b_v&c_v}(0)\ne0.
\end{equation}
Consider the projection $\pi:\M\to\R^2$,
$\pi(u,v,[\alpha:\beta])=(u,v)$.
Then $\pi^{-1}(u,v)$ consists of two points if $b^2-ac>0$,
and is empty if $b^2-ac<0$.
Let us set $p=\beta/\alpha$ 
(we need to consider the case $q=\alpha/\beta$ for some cases)
and
$\F(u,v,p)=ap^2+2bp+c$.
If $\F_p(0)\ne0$, then $\pi$ is a local diffeomorphism,
and
if $\F_p(0)=0$ and $\F_{pp}(0)\ne0$ hold, 
then $\pi$ is a fold
(a map-germ $h:(\R^2,0)\to(\R^2,0)$ is a {\em fold\/} if 
$h$ is $\A$-equivalent to $(u,v)\mapsto(u,v^2)$).
Let us consider the vector field
$$
\xi(u,v,p)=
p\F_p(u,v,p)\partial_u+\F_p(u,v,p)\partial_v
-\big(p\F_u(u,v,p)+\F_v(u,v,p)\big)\partial p.
$$
Then $\xi$ is tangent to $\M$,
and the projections
$d\pi(\xi_1),d\pi(\xi_2)$ satisfy
$$
\omega(d\pi(\xi_i),d\pi(\xi_i))=0
\quad
(i=1,2),
$$
where $\xi_i=\xi(u,v,p_i)$ and 
$\pi^{-1}(u,v)$ $=$ $\{(u,v,p_1),(u,v,p_2)\}$, or
$\pi^{-1}(u,v)$ $=$ $\{(u,v,p_1)\}$.
To study geometric the foliations of $\omega$,
we use $\xi$ on $\M$.
\subsection{Case 1}\label{sec:case1}
We assume that $b^2-ac>0$ at $0$.
Then one can easily see that
the
BDE \eqref{eq:BDE} 
is 
equivalent to 
a  BDE $a\,dv^2+2b\,dudv+c\,du^2=0$
which satisfies $(a(0,0),b(0,0),c(0,0))=(1,0,-1)$.
Furthermore, 
it holds that for any $k>0$,
a BDE $a\,dv^2+2b\,dudv+c\,du^2=0$ 
which satisfies $(a(0,0),b(0,0),c(0,0))=(1,0,-1)$
is equivalent to a BDE $a\,dv^2+2b\,dudv+c\,du^2=0$
whose $k$-jet of $(a,b,c)$ at $(0,0)$ is 
$(1,0,-1)$,
moreover, it is equivalent to
a BDE 
\begin{equation}\label{eq:BDEreg}
\omega_{reg}=dv^2-du^2=0
\end{equation}
(\cite[Proposition 4.4]{BTimplicit}).
Therefore the configuration is a pair of 
transverse smooth foliations.
\subsection{Case 2}\label{sec:case2}
We consider the case 2, namely 
$(a(0),b(0),c(0))\ne(0,0,0)$ and $\delta(0)=0$.
\begin{lemma}\hfill 
We\hfill  assume\hfill  that\hfill 
$(a(0),b(0),c(0))\ne(0,0,0)$\hfill  and\\ $\delta(0)$ $=0$.
If\/ $\omega$ as in\/ \eqref{eq:BDE} satisfies
$$
j^1(a,b,c)(0,0)=(0,0,\alpha_0)
+(\alpha_1v,\alpha_2v, \alpha_3u+\alpha_4 v),
$$
then it 
is equivalent to
a  BDE\/ $a\,dv^2+2b\,dudv+c\,du^2=0$
which satisfies\/ 
\toukouchange{$$(j^1a(0,0),j^1b(0,0),j^1c(0,0))=(1,0,u)$$}
{$(j^1a(0,0),j^1b(0,0),j^1c(0,0))=(1,0,u)$}
when\/ $\alpha_0\alpha_1\ne0$.
\end{lemma}
\begin{proof}
We assume that $\alpha_0\alpha_1\ne0$.
Consider $u=V+\beta_1U^2+\beta_2V^2,\ v=U$
and $R=1+\beta_3 U$, where $\beta_1,\beta_2,\beta_3\in\R$.
Then $R\omega(U,V)$ is given by
$$
\begin{array}{l}
\big(
\alpha_1U+O(2)\big)\,dV^2
+
\big(2(\alpha_2+2\alpha_0\beta_1)U+O(2)\big)\,dUdV\\
\hspace{20mm}
+\big(\alpha_0+(\alpha_4+\alpha_0 \beta_3)U+
(\alpha_3+4\alpha_0 \beta_2)V+O(2)
\big)\,dU^2,
\end{array}
$$
where $O(2)$ stands for remainders of order $2$.
Setting 
$\beta_1=-\alpha_2/(2\alpha_0)$, 
$\beta_2=-\alpha_3/(4\alpha_0)$ and
$\beta_3=-\alpha_4/\alpha_0$,
and re-scaling,
we get the desired result.
\end{proof}
Any BDE of the form \eqref{eq:BDE} 
with $(j^1a(0,0),j^1b(0,0),j^1c(0,0))=
(1,0,u)$ is smoothly equivalent to
\begin{equation}\label{eq:BDEcusp}
\omega_{cusp}=dv^2+u\,du^2=0
\end{equation}
(\cite{dara}, see also \cite[Section 4.2]{BTimplicit}, 
\cite[Proposition 3.3-2]{faridsurvey}).
The solutions form a family of cusps.
\subsection{Case 3}\label{sec:case3}
We assume $(a(0),b(0),c(0))=(0,0,0)$.
In this case, $\delta$ has a critical point at $0$.
We assume $\M$ is a smooth manifold.
Since $\F_p(0,0,p)=0$ holds,
$\xi$ has a zero at $(0,0,p)$ 
if and only if $p\F_u(0,0,p)+\F_v(0,0,p)=0$.
This is a cubic equation for $p$.
Set 
\begin{equation}\label{eq:defphi}
\phi_\omega(p)=\phi(p)=p\F_u(0,0,p)+\F_v(0,0,p).
\end{equation}
Let $D_{\omega}=D$ denotes the discriminant of this equation.
If $D>0$ then $\phi(p)=0$ has three distinct real roots,
and if $D<0$ then $\phi(p)=0$ has one real and
two distinct imaginary roots.

When $D>0$,
let $p_1,p_2,p_3$ be the solutions of $\phi(p)=0$ and
$p_1<p_2<p_3$.
When $D<0$,
let $p_1$ be the solution of $\phi(p)=0$.
If $\phi(0)=\F_v(0,0,0)=a_v(0)\ne0$ then
$p_i\ne0$ ($i=1,2,3$ or $i=1$) holds.
We need to understand the singularity of $\xi$ near $p_i$.
If $\F_u(0,0,p_i)\ne0$,
then $\M$ is parameterized by $v$ as $(u(v,p),v,p)$ near $(0,0,p_i)$.
We denote the linear part of $\xi$ by
$$
j^1\xi(0,0,p_i)=
\big(\xi_{11}v+\xi_{12}(p-p_i)\big)\partial v
+
\big(\xi_{21}v+\xi_{22}(p-p_i)\big)\partial p.
$$
Also we remark that since
$\F(u(v,p),v,p)\equiv0$,
it holds that
$\F_uu_v+\F_v\equiv0$,
where
$\equiv$ means that the equality holds identically.
On the other hand, $p_i$ is a solution of $\phi(p)=0$,
and $\phi(p)=p\F_u+\F_v$,
we have
$$
u_v(0,p_i)=p_i.
$$
Furthermore, $\F_p=0$ at $(u,v)=(0,0)$, it holds that
$u_p=0$ at $(v,p)=(0,p_i)$.
We have
$$
\begin{array}{rcll}
\xi_{11}
&=&\F_{up}u_v+\F_{pv}\\
\xi_{12}
&=&\F_{up}u_p+\F_{pp}\\
&=&0\quad({\rm at}\ (0,0,p_i))
\\
\xi_{22}
&=&-(\F_u+p\F_{uu}u_p+p\F_{up}+\F_{uv}u_p+\F_{vp})\\
\end{array}
$$
Thus the eigenvalues of the linear part of $\xi$
are
\begin{equation}\label{eq:defslphaphiprime}
\alpha(p_i)=
\F_{up}u_v+\F_{pv}\ 
\text{and}\ 
-\phi'(p_i)=
-(\F_u+p\F_{uu}u_p+p\F_{up}+\F_{uv}u_p+\F_{vp}).
\end{equation}
The configuration of the integral curves of $\omega$
is determined by these information.
The following theorem is known.
Let $\det\hess\delta(0,0)$ be
the determinant of the Hesse matrix of $\delta(u,v)$ at $(0,0)$.
\begin{theorem}\label{prop:1}
{\rm \cite[Theorem 4.1]{BTbinary}}
Let\/ $\omega$ 
be a\/ $2$-tensor as in\/ \eqref{eq:BDE} and
satisfies\/ $(a,b,c)(0)=(0,0,0)$,
$\det\hess\delta(0,0)<0$,
$D\ne0$
and\/ $\phi(p)$ and\/ $\alpha(p)$ do not have any common roots.
Then the BDE\/ $\omega=0$ is topologically equivalent to
one of the following BDEs\/{\rm :}
\begin{itemize}
\item The case\/ $D>0$\,{\rm :}
$($Then\/ $\phi(p)=0$ has\/ $3$ roots\/ $p_1,p_2,p_3$.$)$
\begin{itemize}
\item $\omega_{3s}=vdv^2+2ududv+vdu^2=0$ $(3$ saddles\/$)$
when\\ $-\phi'(p_i)\alpha(p_i)$ are
negative for all\/ $i=1,2,3$.
\item $\omega_{2s1n}=vdv^2+2(v-u)dudv+vdu^2=0$ $(2$ saddles\/ $+$ $1$ node\/$)$
when\/ $-\phi'(p_i)\alpha(p_i)$ are
two negative and one positive for\/ $i=1,2,3$.
\item $\omega_{1s2n}=3vdv^2-4ududv+3vdu^2=0$ $(1$ saddle\/ $+$ $2$ nodes\/$)$
when\/ $-\phi'(p_i)\alpha(p_i)$ are
one negative and two positive for\/ $i=1,2,3$.
\end{itemize}
\item The case\/ $D<0$\,{\rm :} 
$($Then\/ $\phi(p)=0$ has\/ $1$ root\/ $p_1$.$)$
\begin{itemize}
\item $\omega_{1s}=vdv^2+2ududv+vdu^2=0$ $(1$ saddle\/$)$
when\\ $-\phi'(p_1)\alpha(p_1)$ is
positive.
\item $\omega_{1n}=2vdv^2-ududv+2vdu^2=0$ $(1$ node\/$)$
when\\ $-\phi'(p_1)\alpha(p_1)$ is
negative.
\end{itemize}
\end{itemize}
\end{theorem}
Note that in the case of $D>0$,
all 
$-\phi'(p_i)\alpha(p_i)$ $i=1,2,3$ 
cannot be positive, see \cite{BTbinary}.
The integral curves of the above 
BDEs are in Figures \ref{fig:1}, \ref{fig:2}, \ref{fig:3}
which are taken from \cite{BTbinary}.

\begin{figure}[!ht]
\centering
\includegraphics[width=.25\linewidth]{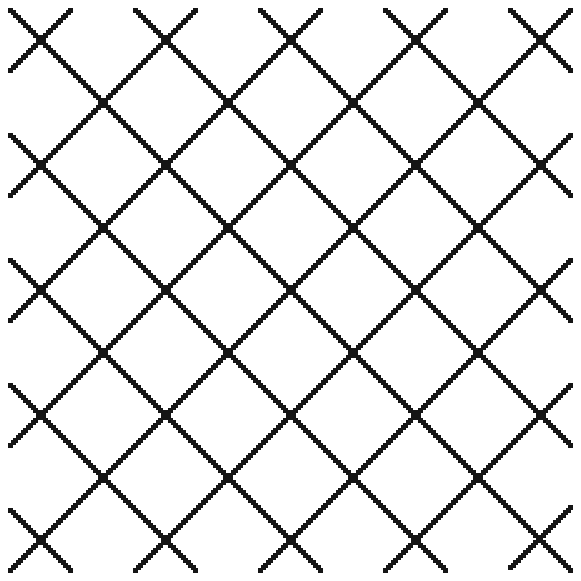}
\hspace{1mm}
\includegraphics[width=.25\linewidth]{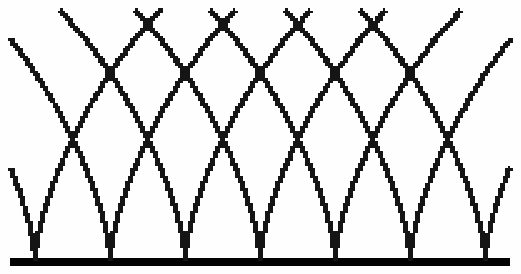}
\caption{Integral curves of
$\omega_{reg}$ and $\omega_{cusp}$.}
\label{fig:1}
\end{figure}

\begin{figure}[!ht]
\centering
\includegraphics[width=.25\linewidth]{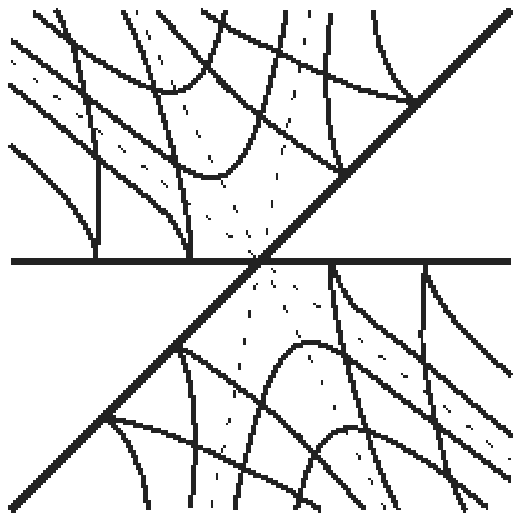}
\hspace{1mm}
\includegraphics[width=.25\linewidth]{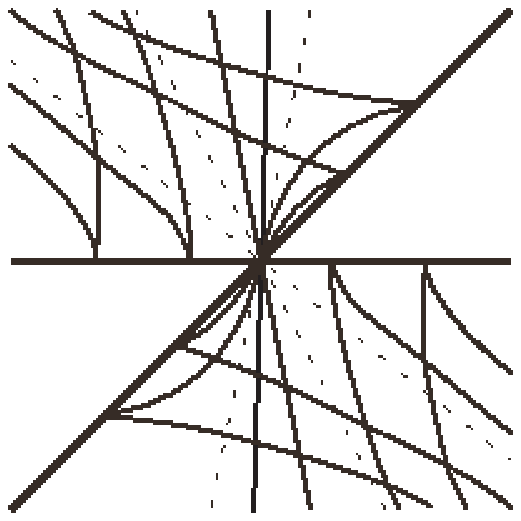}
\hspace{1mm}
\includegraphics[width=.25\linewidth]{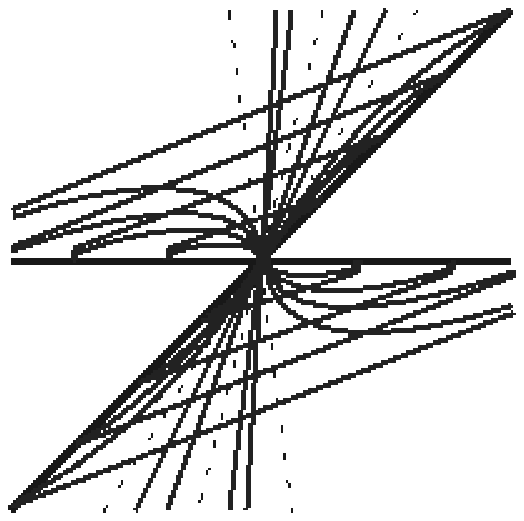}
\caption{Integral curves of
$\omega_{3s}$, $\omega_{2s1n}$ and $\omega_{1s2n}$.
Here and in the rest of the paper, 
the dashed curves are separatrices, 
i.e., they are integral curves 
that separate distinct topological sectors. 
}
\label{fig:2}
\end{figure}

\begin{figure}[!ht]
\centering
\includegraphics[width=.25\linewidth]{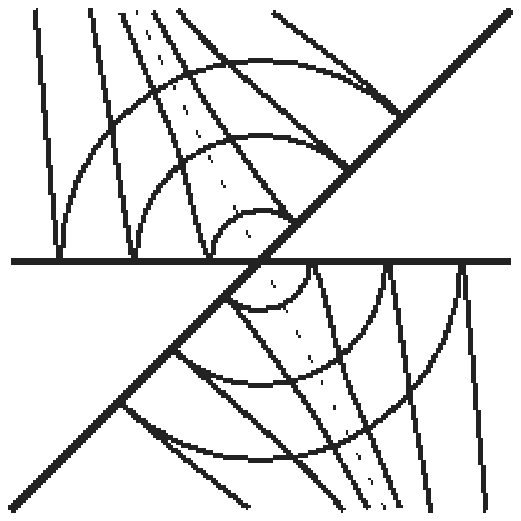}
\hspace{1mm}
\includegraphics[width=.25\linewidth]{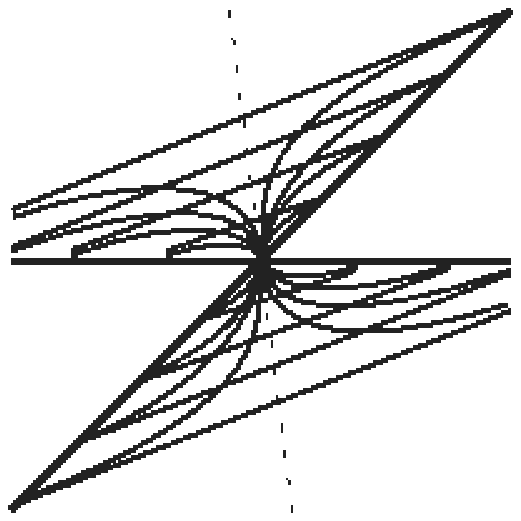}
\caption{Integral curves of
$\omega_{1s}$ and $\omega_{1n}$.}
\label{fig:3}
\end{figure}

\section{Geometric binary differential equations on a cuspidal edge}
Let $f:(\R^2,0)\to(\R^3,0)$ be a parametrization of 
a cuspidal edge.
We take $f$ as in \eqref{eq:west2}.
Then the coefficients of the first fundamental form 
of the cuspidal edge with respect to $f$ are
\begin{equation} \label{coeff:1ff1}
\begin{array}{rcl}
E
&=&
 1+(a_{20}^2 + b_{20}^2)u^2 
+(a_{20} a_{30}+b_{20} b_{30})u^3 
+  b_{12} b_{20} uv^2 \\[2mm]
& & \hspace{0mm}
+ \dfrac{1}{4}\Big(a_{30}^2 + b_{30}^2 + 2 a_{20} p(0)
+2 b_{20} q_1(0)\Big) u^4 \\[3mm]
 & & \hspace{0mm}
+\dfrac{1}{2}\Big(b_{12} b_{30}+8 b_{20}q_2(0)\Big) u^2v^2
+2b_{20}q_3(0) uv^3\\[3mm]
 & & \hspace{0mm}
+\dfrac{1}{4}b_{12}^2 v^4 + O(5),
\end{array}
\end{equation}
\begin{equation} \label{coeff:1ff2}
\begin{array}{rcl}
F
&=& 
 a_{20}uv 
+ \left(\dfrac{1}{2}a_{30}+ b_{12}b_{20}\right)u^2v 
+ \dfrac{1}{2}b_{03}b_{20}uv^2\\[3mm]
& &
+ \dfrac{1}{4}b_{03} b_{12} v^4 
+\left(\dfrac{1}{2}b_{12}^2+4 b_{20} q_4(0,0)\right) u v^3\\[3mm]
& &
+\dfrac{1}{4}\Big(b_{03} b_{30}+12 b_{20} q_3(0)\Big)u^2 v^2\\[3mm]
& & 
\hspace{0mm}+ 
\left(\dfrac{1}{2}b_{12} b_{30}+4p(0)+2 b_{20} q_2(0)\right) u^3v 
+ O(5), \\[3mm]
\end{array}
\end{equation}
\begin{equation} \label{coeff:1ff3}
\begin{array}{rcl}
G
&=& 
v^2+ \dfrac{1}{4}b_{03}^2 v^4+b_{03} b_{12} u v^3+b_{12}^2 u^2v^2
+ O(5),
\end{array}
\end{equation}
where $O(n)$ stands for remainders of order $n$ 
$(n=1,2,\ldots)$.
Since $f$ is as in \eqref{eq:west2},
we see $f_u\times (f_v/v)\ne0$.
We set $\nu_2=f_u\times (f_v/v)$,
and 
$L_2=\inner{f_{uu}}{\nu_2}$, 
$M_2=\inner{f_{uv}}{\nu_2}$, 
$N_2=\inner{f_{vv}}{\nu_2}$.
Then we have:
\begin{equation} \label{coeff:2ff}
\begin{array}{rcl}
L_2
&=&
b_{20}+ (b_{30}- a_{20} b_{12}) u -\dfrac{1}{2}a_{20} b_{03} v\\[3mm]
&&\hspace{0mm}
-\big(a_{30} b_{12}-12 q_1(0)+2 a_{20} q_2(0)\big) u^2\\[2mm]
&& \hspace{0mm}
-\big(a_{30} b_{03}+6 a_{20} q_3(0)\big)u v\\[2mm]
&& \hspace{0mm}
+ \big(2 q_2(0)-4 a_{20} q_4(0,0)\big) v^2
+ O(3) \\[2mm]
M_2
&=&
b_{12} v+4  q_2(0) u v + 3 q_3(0) v^2 + O(3)\\[2mm]
N_2
&=&
\dfrac{1}{2}b_{03} v+3 q_3(0)uv+ 8 q_4(0,0) v^2 + O(3).
\end{array}
\end{equation}
We have $L_2=L|\nu_2|$, $M_2=M|\nu_2|$, $N_2=N|\nu_2|$.
It should be remarked that
there exist $C^\infty$-functions
$\widetilde F,\widetilde G,\widetilde N,\widetilde M$
such that
$
G=v^2 \widetilde G$,
$F=v \widetilde F$,
$N_2=v \widetilde N$ and
$M_2=v \widetilde M$
holds. We set 
\begin{equation}\label{eq:fundform}
\widetilde E=E,\ 
\widetilde F=\dfrac{F}{v},\ 
\widetilde G=\dfrac{G}{v^2},\ 
\widetilde L=L_2,\ 
\widetilde M=\dfrac{M_2}{v},\ 
\widetilde N=\dfrac{N_2}{v}.
\end{equation}

\subsection{Lines of principal curvature}\label{sec:lc}
In this subsection we consider the BDE $\omega_{lc}=0$.
Using \eqref{eq:fundform},
$\omega_{lc}=0$ is equivalent to
$$
v^2\left(\widetilde F\widetilde N - v\widetilde G\widetilde M\right)dv^2 
+v\left(\widetilde E\widetilde N-v^2\widetilde  G\widetilde  L\right)dudv
+v\left(\widetilde E\widetilde M-\widetilde F\widetilde  L\right)du^2=0.$$
To determine the topological configuration of $\omega_{lc}=0$,
we factor out $v$ and consider $\widetilde\omega_{lc}=0$, where
$$
\widetilde\omega_{lc}
=
v\left(\widetilde F\widetilde N - v\widetilde G\widetilde M\right)dv^2 
+\left(\widetilde E\widetilde N-v^2\widetilde  G\widetilde  L\right)dudv
+\left(\widetilde E\widetilde M-\widetilde F\widetilde  L\right)du^2.
$$
We have the following proposition.
\begin{proposition}\label{prop:lc}
The BDE\/ $\widetilde\omega_{lc}=0$ is equivalent to 
the BDE\/ $\omega_{reg}=0$.
\end{proposition}
This proposition implies that
the lines of principal curvature of
a cuspidal edge
form a pair of smooth and transverse foliations
in the domain of a parametrization.
\begin{proof}[Proof of Proposition\/ {\rm \ref{prop:lc}}]
Set
$$
(a,b,c)=\left(
v\widetilde F\widetilde N - v^2 \widetilde G \widetilde M,\ 
\dfrac{1}{2}\Big(\widetilde E \widetilde N
-v \widetilde G \widetilde L\Big),\ 
\widetilde E \widetilde M-\widetilde F\widetilde L\right).
$$
Then since 
$b(0)=\widetilde E(0)\widetilde N(0)$,
$\widetilde E(0)=\inner{f_u}{f_u}(0)\ne0$
and
$\widetilde N(0)=\inner{f_{vvv}}{\nu_2}(0)=b_{03}\ne0$
hold, $\widetilde \omega_{lc}$ is as in Case 1.
Moreover, since $a(0)=0$ and $b(0)\ne0$,
we see that $b^2-ac>0$ at $0$.
Hence $\widetilde\omega_{lc}$ is equivalent
to $\omega_{reg}=0$ (See Section \ref{sec:case1}).
\end{proof}
The fact of
the existence of the curvature 
line coordinate system at cuspidal edge
is also shown in \cite{mu}.

An example of picture of this configuration on 
the cuspidal edges
is in Figures \ref{fig:prin}.
Since one family of the integral curves 
are tangent to the null direction on
singular curve,
one family of the integral curves 
near singular curve form the
$(2,3)$-cusps.
A map-germ $(\R,0)\to(\R^3,0)$ is an {\it $(2,3)$-cusp\/}
if it is $\A$-equivalent to $t\mapsto(t^2,t^3)$.

\begin{figure}[!ht]
\centering
\includegraphics[width=.25\linewidth]{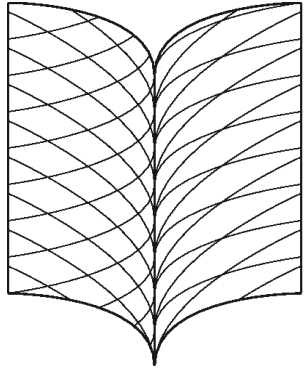}
\caption{Integral curves of
$\widetilde\omega_{lc}$
on images of cuspidal edges.}
\label{fig:prin}
\end{figure}


\subsection{Asymptotic curves}
In this subsection we consider $\omega_{as}=0$,
and it is equivalent to
$\widetilde\omega_{as}=N_2\,dv^2+2M_2\,dudv+N_2\,du^2=0$.
Since
$M_2(u,0)=0$, and
$N_2(u,0)=0$,
the singular set i.e., the cuspidal edge curve
is part of discriminant of $\widetilde\omega_{as}=0$,
and, $\partial_v$ is a solution to $\widetilde\omega_{as}=0$
on the singular set.
By \eqref{coeff:2ff},
\begin{equation}\label{eq:delta}
\begin{array}{rcl}
\displaystyle \delta(u,v) &=&
\displaystyle
-\dfrac{1}{2}b_{20}b_{03}v
+\left(\frac{1}{4}a_{20} b_{03}^2 + b_{12}^2
- 8 b_{20}q_4(0,0)\right) v^2 \\[3mm]
&&
+
\dfrac{1}{2}
b_{03}\Big(a_{20} b_{12} - b_{30}-2b_{20}q_3(0)\Big)uv\\[3mm]
&&
\displaystyle
+
\Big(-b_{03}q_2(0) + 6 b_{12} q_3(0) \\[3mm]
&&
\displaystyle
\hspace{15mm}
+ 6 a_{20} b_{03} q_4(0,0)-7b_{20}(q_4)_v(0,0)\Big)v^3\\[3mm]
&&\displaystyle
+
\Bigg(
\frac{1}{4}a_{30} b_{03}^2 + 3 a_{20} b_{03} q_3(0)
-
8 b_{30} q_4(0,0) \\[3mm]
&&\hspace{15mm}
+ 8 b_{12}q_2(0) 
+ 8 a_{20}b_{12} q_4(0,0)
\Bigg)uv^2\\[3mm]
&&
\displaystyle
+
\Bigg(\frac{1}{2}a_{30} b_{03} b_{12}
- 6 b_{03} q_1(0) + a_{20} b_{03} q_2(0)
\\[3mm]
&&\hspace{15mm}
+
3 a_{20} b_{12} q_3(0) - 3 b_{30} q_3(0)\Bigg)u^2v
+
O(4).
\end{array}
\end{equation}
Thus if $b_{20}\ne0$, then 
the BDE is in Case 2.
In this case, by \eqref{coeff:2ff},
$\omega_{as}=0$
is equivalent to $\omega_{cusp}=0$
(see Subsection \ref{sec:case2}).
By \cite[Corollary 3.6]{front}, 
$b_{20}\ne0$ implies that
the Gaussian curvature 
is unbounded and changes sign between 
the two sides of the cuspidal edge.
This means that in this case, the singular set of
cuspidal edge plays the same role as the 
parabolic curve on regular surfaces.
Since 
$b_{20}\ne0$, then 
the BDE is $\omega_{cusp}=0$,
this implies that the folded saddle, the
folded node, the folded
focus in Davydov's classification \cite{d1}
(see also \cite[Section 3.2]{faridsurvey})
does not appear not only cuspidal edges,
but also all singularities written in the form
\eqref{eq:west1} (for instance, cuspidal cross cap).

We assume now $b_{20}=0$. 
Then the BDE is in Case 3.
We have the following.
If $a_{20} b_{12} - b_{30}\ne0$, then
$\delta$ is a Morse function near $0$.
We study the geometric foliation near $0$ 
as in Subsection \ref{sec:case3}.
We consider
$$
\F(u,v,p)=
N_2+2M_2p+L_2p^2.
$$
Then we have
$$
\begin{array}{rcl}
(\F_u,\F_v,\F_p)(u,v,p)
&=&
\big( (N_2)_u+2 (M_2)_up+ (L_2)_up^2,\\
&&\hspace{0mm}
 (N_2)_v+2 (M_2)_vp+ (L_2)_vp^2,
2 M_2+2 L_2p\big),\\
(\F_u,\F_v,\F_p)(0,0,p)
&=&
\big( (L_2)_up^2,
 (N_2)_v+2 (M_2)_vp+ (L_2)_vp^2,
2 Lp\big)\\
&\hspace{-14mm}=&\hspace{-10mm}
\left(\big(b_{30}-a_{20} b_{12}\big) p^2,
\dfrac{1}{2}\big(b_{03}+4 b_{12} p-a_{20} b_{03} p^2\big),
0\right).
\end{array}
$$
In this case, the
left hand side of
\eqref{eq:mmfd}
is $(b_{30}-a_{20}b_{12})^2b_{03}^2/4$.
Thus if $b_{30}-a_{20}b_{12}\ne0$ at $0$,
then $\M=\{\F=0\}$ is a smooth manifold.
We have $\phi_{as}=\phi_{\omega_{as}}$ 
and $D_{as}=D_{\omega_{as}}$ 
defined in Subsection \ref{sec:case3}
as follows
$$
\begin{array}{rcL}
\phi_{as}(p)&=&
 (b_{30}-a_{20} b_{12}) p^3
-\frac{1}{2} a_{20} b_{03} p^2
+2 b_{12} p
+\frac{1}{2}b_{03},\\[2mm]
4D_{as}
&=&
 a_{20}^3 b_{03}^4
+13 a_{20}^2 b_{03}^2 b_{12}^2
-b_{30} (128 b_{12}^3+27 b_{03}^2 b_{30})\\
&&
+2 a_{20} (64 b_{12}^4+9 b_{03}^2 b_{12} b_{30}).
\end{array}
$$
Furthermore, $\alpha(p)$ 
is given by
$$
\alpha(p)=
2 (b_{30}-a_{20} b_{12}) p^2-a_{20} b_{03} p+2 b_{12}.
$$
If $p_i$ is a solution of $\phi_{as}(p)=0$ and
$p\alpha(p)-2\phi_{as}(p)=-b_{03} - 2 b_{12} p$ holds,
then $\alpha(p_i)\ne0$ if and only if $-b_{03} - 2 b_{12} p_i\ne0$.
Assume that $b_{12}\ne0$.
Substituting $p=-b_{03}/(2 b_{12})$ into $\phi_{as}(p)$,
we get
$$
\phi_{as}\left(\dfrac{-b_{03}}{2 b_{12}}\right)
=
-\dfrac{1}{8b_{12}^3} b_{03} \big(4 + b_{03}^2 b_{30}\big).
$$
If $b_{12}=0$, then $p\alpha(p)-2\phi_{as}(p)=-b_{03}\ne0$.
Since we assume that $b_{30}-a_{20} b_{12}\ne0$,
 if $b_{12}=0$ then $b_{30}\ne0$.
Thus we get
$\alpha(p_i)\ne0$ if and only if
$$
4b_{12}^3 + b_{03}^2 b_{30}\ne0.
$$
We can now use Theorem \ref{prop:1}
to obtain the following result.
\begin{proposition}\label{prop:asym}
If\/ $b_{20}\ne0$, 
then\/ $\omega_{as}$ is equivalent
to\/ $\omega_{cusp}=0$.

If\/ $b_{20}=0$, $b_{30}-a_{20} b_{12}\ne0$,
$D_{as}\ne0$ and\/ $4b_{12}^3 + b_{03}^2 b_{30}\ne0$,
then\/ $\omega_{as}$ is topologically equivalent
to one of the following\/$:$
\begin{itemize}
\item The case\/ $D_{as}>0$\,{\rm :} 
$($Then\/ $\phi_{as}(p)=0$ has\/ $3$ roots\/ $p_1,p_2,p_3)$
\begin{itemize}
\item $\omega_{3s}$ $($$-\phi_{as}'(p_i)\alpha(p_i)$ are
negative for all\/ $i=1,2,3)$.
\item $\omega_{2s1n}$ $($$-\phi_{as}'(p_i)\alpha(p_i)$ are
two negative and one positive for\/ $i=1,2,3)$.
\item $\omega_{1s2n}$ $($$-\phi_{as}'(p_i)\alpha(p_i)$ are
one negative and two positive for\/ $i=1,2,3)$.
\end{itemize}
\item The case $D_{as}<0$\,{\rm :} 
$($Then\/ $\phi_{as}(p)=0$ has\/ $1$ root\/ $p_1)$
\begin{itemize}
\item $\omega_{1s}$ $($$-\phi_{as}'(p_1)\alpha(p_1)$ is
negative\/$)$.
\item $\omega_{1n}$ $($$-\phi_{as}'(p_1)\alpha(p_1)$ is
positive\/$)$.
\end{itemize}
\end{itemize}
\end{proposition}
\begin{remark}
We observe that
by the Proposition \ref{prop:asym},
$b_{20}$, $b_{30}-a_{20} b_{12}$ and $4b_{12}^3 + b_{03}^2 b_{30}$
have geometric meanings.
In fact, $b_{20}$ is the limiting normal curvature
and $b_{30}-a_{20} b_{12}$ coincides with the derivation 
of limiting normal curvature (see \cite{MSUY}).
The invariant $4b_{12}^3 + b_{03}^2 b_{30}$ is related to
the singularities of parallel surfaces of the cuspidal edge
(see \cite[Lemma 3.1]{teramoto}).
\end{remark}

\subsection{Characteristic curves}
We consider the BDE $\omega_{ch}=0$.
Using \eqref{eq:fundform},
we show that $\omega_{ch}=0$ is equivalent to
$v(a\,dv^2+2b\,dudv+c\,du^2)=0$,
where
$$
\begin{array}{rcL}
a&=&
v \Big(\widetilde E \widetilde N^2 + 
   \big(-\widetilde G \widetilde L \widetilde N 
     - 2 \widetilde F \widetilde M \widetilde N\big)v
     + 2 \widetilde G \widetilde M^2 v^2\Big)\\
b&=&
 v \big(-2 \widetilde F \widetilde L \widetilde N 
     + \widetilde E \widetilde M \widetilde N 
     + \widetilde G \widetilde L \widetilde M v\big)\\
c&=&
-\widetilde E \widetilde L \widetilde N+
\big(\widetilde G \widetilde L^2 
- 2 \widetilde F \widetilde L\widetilde M
+2 \widetilde E\widetilde M^2
\big) v.
\end{array}
$$
We factor out $v$, so $\omega_{ch}=0$
is topologically equivalent to 
$\tilde\omega_{ch}=a\,dv^2+2b\,dudv+c\,du^2=0$.
Since
$ a(u,0)=0$, and
$ b(u,0)=0$,
the singular set is a part of discriminant of $\tilde\omega_{ch}=0$,
and $\partial_v$ is a solution to $\tilde\omega_{ch}=0$
on the singular set.
The function $\delta=b^2-ac$ is given by
$$
\begin{array}{l}
v\Big[\widetilde E^2 \widetilde L \widetilde N^3 
+\big(4 \widetilde F^2 \widetilde L^2 
-2 \widetilde E \widetilde G \widetilde L^2
-4 \widetilde E \widetilde F \widetilde L \widetilde M
-2 \widetilde E^2 \widetilde M^2\big) \widetilde N^2v \\
\hspace{5mm}
+
\big(
- \widetilde G^2 \widetilde L^3 
-4\widetilde F \widetilde G \widetilde L^2 \widetilde M
-4\widetilde F^2 \widetilde L \widetilde M^2
+6\widetilde E \widetilde G \widetilde L \widetilde M^2
+4\widetilde E \widetilde F \widetilde M^3\big) \widetilde Nv^2 \\
\hspace{10mm}
+
\big(
  \widetilde G \widetilde L^2
-4\widetilde F \widetilde L \widetilde M 
+4\widetilde E \widetilde M^2\big) \widetilde G \widetilde M^2v^3\Big].
\end{array}
$$
When $f$ is taken as in \eqref{eq:west2},
we have
$$
\begin{array}{RcL}
a&=&\frac{1}{4}b_{03}^2v+O(2),\\[2mm]
b&=&\frac{1}{2}b_{12}b_{03}v+O(2),\\[2mm]
c&=&-\frac{1}{2}b_{20}b_{03}
+\frac{1}{2}\big(a_{20} b_{12} b_{03}
             -b_{30} b_{03} -6 b_{20} q_3(0)\big)u\\[2mm]
&&\hspace{10mm}
+\frac{1}{4}\Big[
a_{20} b_{03}^2+8 b_{12}^2+4b_{20} \big(b_{20}-8 q_4(0,0)\big)
\Big]v
+O(2),\\[2mm]
\delta&=&
\dfrac{1}{8}b_{03}^3 b_{20}v
+
\dfrac{1}{8}b_{03}^2\big(-a_{20} b_{03} b_{12}
        +b_{03} b_{30}+18 b_{20} q_3(0)\big)u v \\[2mm]
&&\hspace{1mm}
-
\dfrac{1}{16}b_{03}^2\bigg[a_{20} b_{03}^2+4 
\Big[b_{12}^2+2 b_{20} \big(b_{20}-12 q_4(0,0)\big)\Big]\bigg]
 v^2 +O(3).
\end{array}
$$
Since $b_{03}\ne0$, 
if $b_{20}\ne0$, then $\widetilde\omega_{ch}$
is as in Case 2, and
if $b_{20}=0$, then it is as in Case 3.
In the following, we divide our consideration into these two cases.
\subsubsection*{The case\/ $b_{20}\ne0$ {\rm :}}
By the argument in Subsection \ref{sec:case2},
$\widetilde\omega_{ch}=0$ is equivalent to $\omega_{cusp}=0$.
In this case, by \eqref{coeff:2ff},
$\omega_{as}=0$
is equivalent to $\omega_{cusp}=0$
(see Subsection \ref{sec:case2}).
Like as the case of $\omega_{as}=0$,
the singular set of
cuspidal edge plays the same role as the 
parabolic curve on regular surfaces.
Moreover, the folded saddle, the
folded node, the folded
focus do not appear.

\subsubsection*{The case\/ $b_{20}=0$ {\rm :}}
The left hand side of \eqref{eq:mmfd} is
$b_{03}^6(a_{20}b_{12}-b_{30})^2/64$.
Thus $\M$ is a smooth manifold if $a_{20}b_{12}-b_{30}\ne0$ at $0$.
Furthermore,
 $\delta$ is of Morse type if and only if $a_{20} b_{12}-b_{30}\ne0$.
This is exactly the same conditions 
as the case of asymptotic curves.
We assume that $a_{20} b_{12}-b_{30}\ne0$.
We need to consider
$\F(u,v,p)=a+2bp+cp^2$
and
$\phi_{ch}(p)=(p\F_u+\F_v)(0,0,p)$.
We have
\begin{equation}
\label{eq:charphi}
\begin{array}{l}
4\phi_{ch}(p)
=
b_{03}^2 
+
4 b_{03} b_{12} p \\
\hspace{18mm}
+ (a_{20} b_{03}^2 + 8 b_{12}^2) p^2 
+ (2 a_{20} b_{03} b_{12} - 2 b_{03} b_{30}) p^3,
\end{array}
\end{equation}
and
the discriminant $D_{ch}=D_{\omega_{ch}}$ of the cubic $\phi_{ch}$ is 
given by
\begin{equation}
\label{eq:chardelta1}
\begin{array}{L}
D_{ch}
=-\dfrac{b_{03}^2}{64}\Big(
a_{20}^3 b_{03}^6+11 a_{20}^2 b_{03}^4 b_{12}^2
\\
\hspace{20mm}
-2 a_{20} (16 b_{03}^2 b_{12}^4+9 b_{03}^4 b_{12} b_{30})
\\
\hspace{40mm}
+256 b_{12}^6
+160 b_{03}^2 b_{12}^3 b_{30}+27 b_{03}^4 b_{30}^2
\Big).
\end{array}
\end{equation}
Furthermore, $4\alpha$ is given by
$$
4\alpha(p)=
2b_{03}(a_{20}b_{12}-b_{30})p^2+(a_{20}b_{03}^2+8b_{12}^2)p+2b_{12}b_{03}.
$$
Then we have
\begin{equation}
\label{eq:charphipalpha}
\begin{array}{l}
4\phi_{ch}'(p)\alpha(p)=
4 b_{03}^2 b_{12}^2
+4 b_{03} b_{12} (a_{20} b_{03}^2+8 b_{12}^2) p\\
\hspace{22mm}
+(a_{20}^2 b_{03}^4+26 a_{20} b_{03}^2 b_{12}^2
 +64 b_{12}^4-10 b_{03}^2 b_{12} b_{30}) p^2\\
\hspace{22mm}
 +5 b_{03} (a_{20} b_{03}^2+8 b_{12}^2)
 (a_{20} b_{12}-b_{30}) p^3\\
\hspace{22mm}
 +6 b_{03}^2 (-a_{20} b_{12}+b_{30})^2 p^4,
\end{array}
\end{equation}
and the condition that $\phi_{ch}(p)$ and $\alpha(p)$ do
not have any common roots is given by
$4b_{12}^3 + b_{03}^2 b_{30}\ne0$.
We summerize the above discussion in the following
proposition.
\begin{proposition}\label{prop:char}
If\/ $b_{20}\ne0$, 
then\/ $\omega_{ch}$ is equivalent
to\/ $\omega_{cusp}=0$.

If\/ $b_{20}=0$, $b_{30}-a_{20} b_{12}\ne0$,
$D_{ch}\ne0$ and\/ $4b_{12}^3 + b_{03}^2 b_{30}\ne0$,
then\/ $\omega_{ch}$ is topologically equivalent
to one of the following\/{\rm :}
\begin{itemize}
\item The case\/ $D_{ch}>0$\,{\rm :} 
$($Then\/ $\phi_{ch}(p)=0$ has\/ $3$ roots\/ $p_1,p_2,p_3)$
\begin{itemize}
\item $\omega_{3s}$ $($$-\phi_{ch}'(p_i)\alpha(p_i)$ are
negative for all\/ $i=1,2,3)$.
\item $\omega_{2s1n}$ $($$-\phi_{ch}'(p_i)\alpha(p_i)$ are
two negative and one positive for\/ $i=1,2,3)$.
\item $\omega_{1s2n}$ $($$-\phi_{ch}'(p_i)\alpha(p_i)$ are
one negative and two positive for\/ $i=1,2,3)$.
\end{itemize}
\item The case $D_{ch}<0$\,{\rm :} 
$($Then\/ $\phi_{ch}(p)=0$ has\/ $1$ root\/ $p_1)$
\begin{itemize}
\item $\omega_{1s}$ $($$-\phi_{ch}'(p_1)\alpha(p_1)$ is
negative\/$)$.
\item $\omega_{1n}$ $($$-\phi_{ch}'(p_1)\alpha(p_1)$ is
positive\/$)$.
\end{itemize}
\end{itemize}
\end{proposition}
\begin{remark}
We remark that if $b_{12}=0$, then $D_{as}=D_{ch}$.
Namely, the configrations of foliations with respect to
$\omega_{as}$ and $\omega_{ch}$ are of the same type.
\end{remark}

Examples of pictures of these configrations on 
the cuspidal edges
are in Figures \ref{fig:cecusp},
\ref{fig:ce2} and \ref{fig:ce3}.
Since the integral curves emanate from singular curve
along the null direction,
integral curves near the singular curve do not form the
$(2,3)$-cusp but form the $(3,4)$-cusps
(see Appendix \ref{sec:cri}).

\begin{figure}[!ht]
\centering
\includegraphics[width=.6\linewidth]{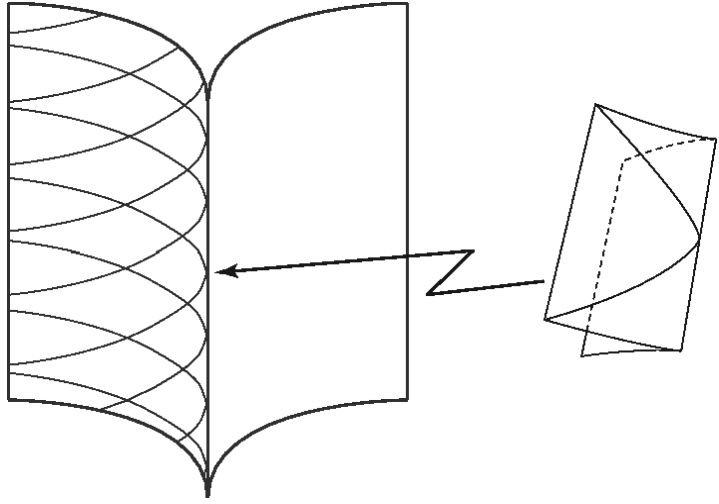}
\caption{Integral curves of
$\omega_{cusp}$
on images of cuspidal edges.}
\label{fig:cecusp}
\end{figure}
\begin{figure}[!ht]
\centering
\includegraphics[width=.3\linewidth]{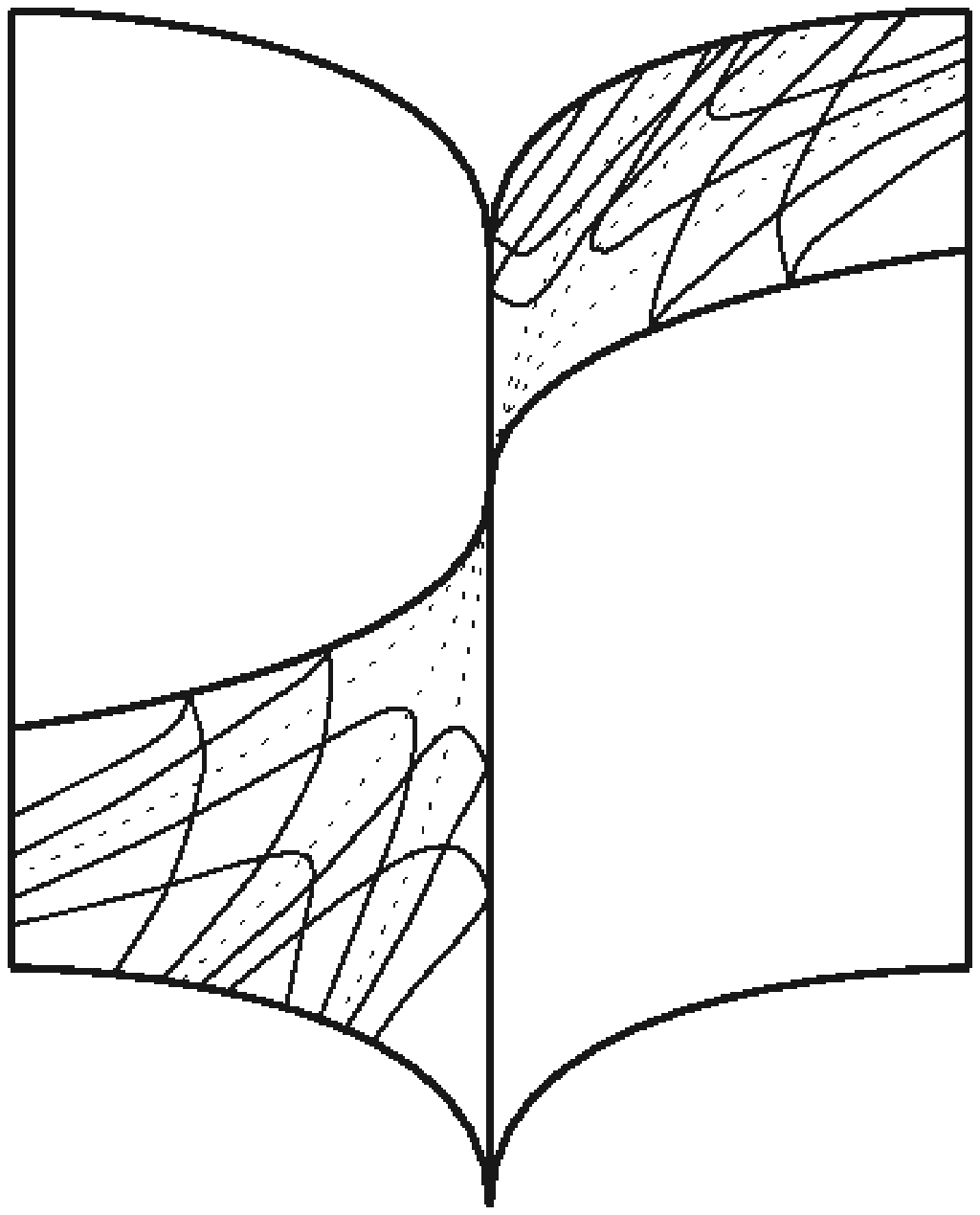}
\hspace{1mm}
\includegraphics[width=.3\linewidth]{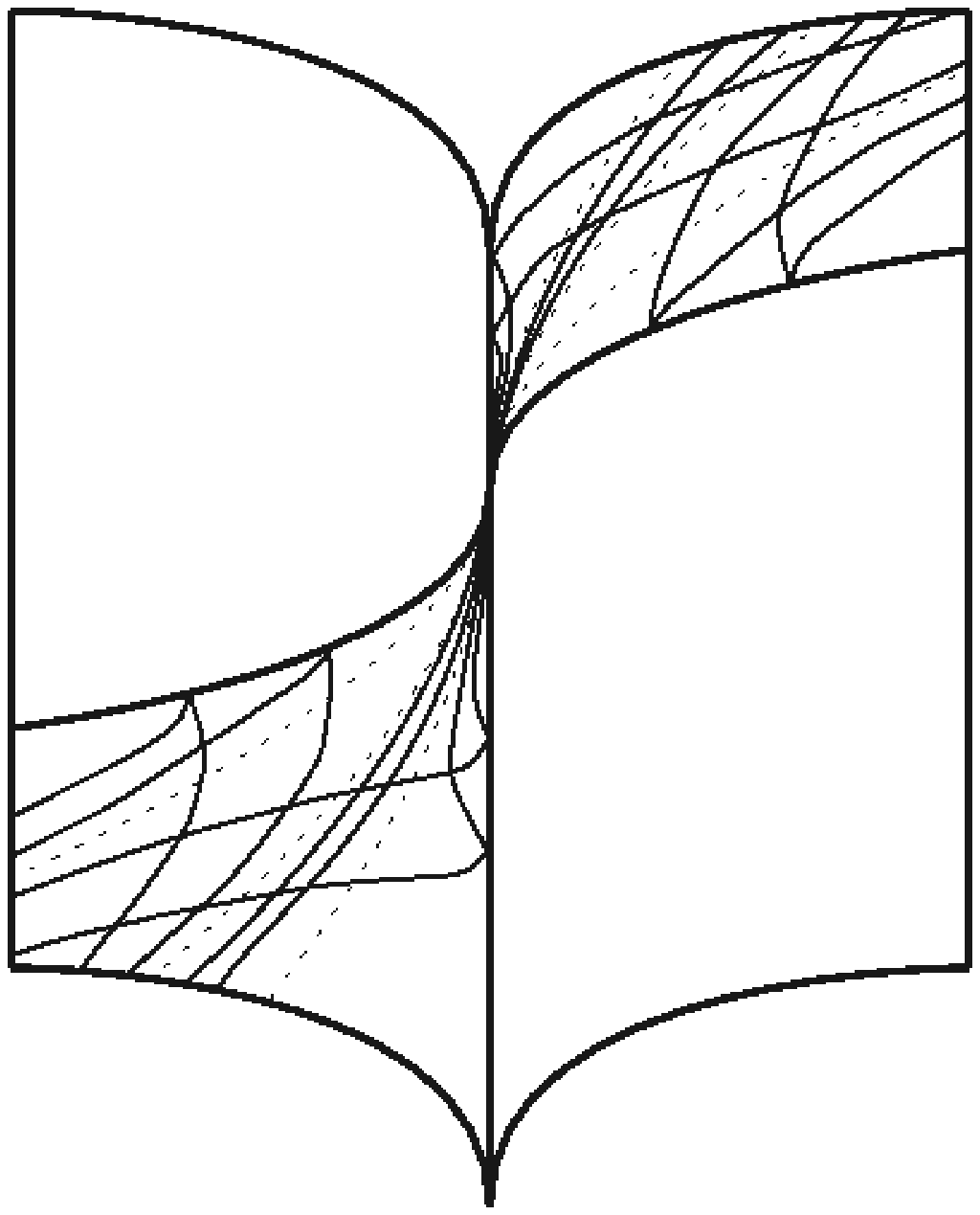}
\hspace{1mm}
\includegraphics[width=.3\linewidth]{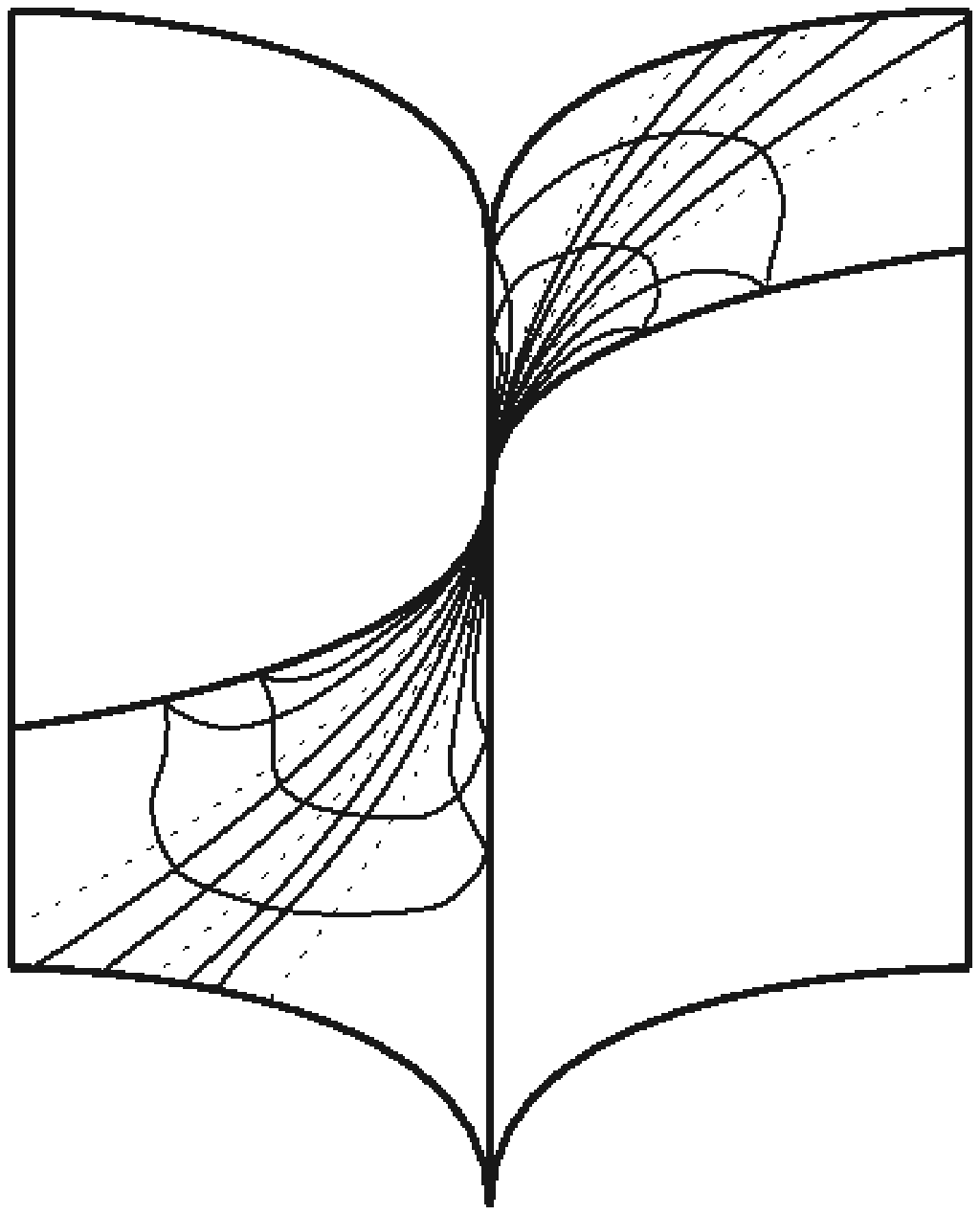}
\caption{Integral curves of
$\omega_{3s}$,  $\omega_{2s1n}$ and $\omega_{1s2n}$ 
on images of cuspidal edges.}
\label{fig:ce2}
\end{figure}
\begin{figure}[!ht]
\centering
\includegraphics[width=.3\linewidth]{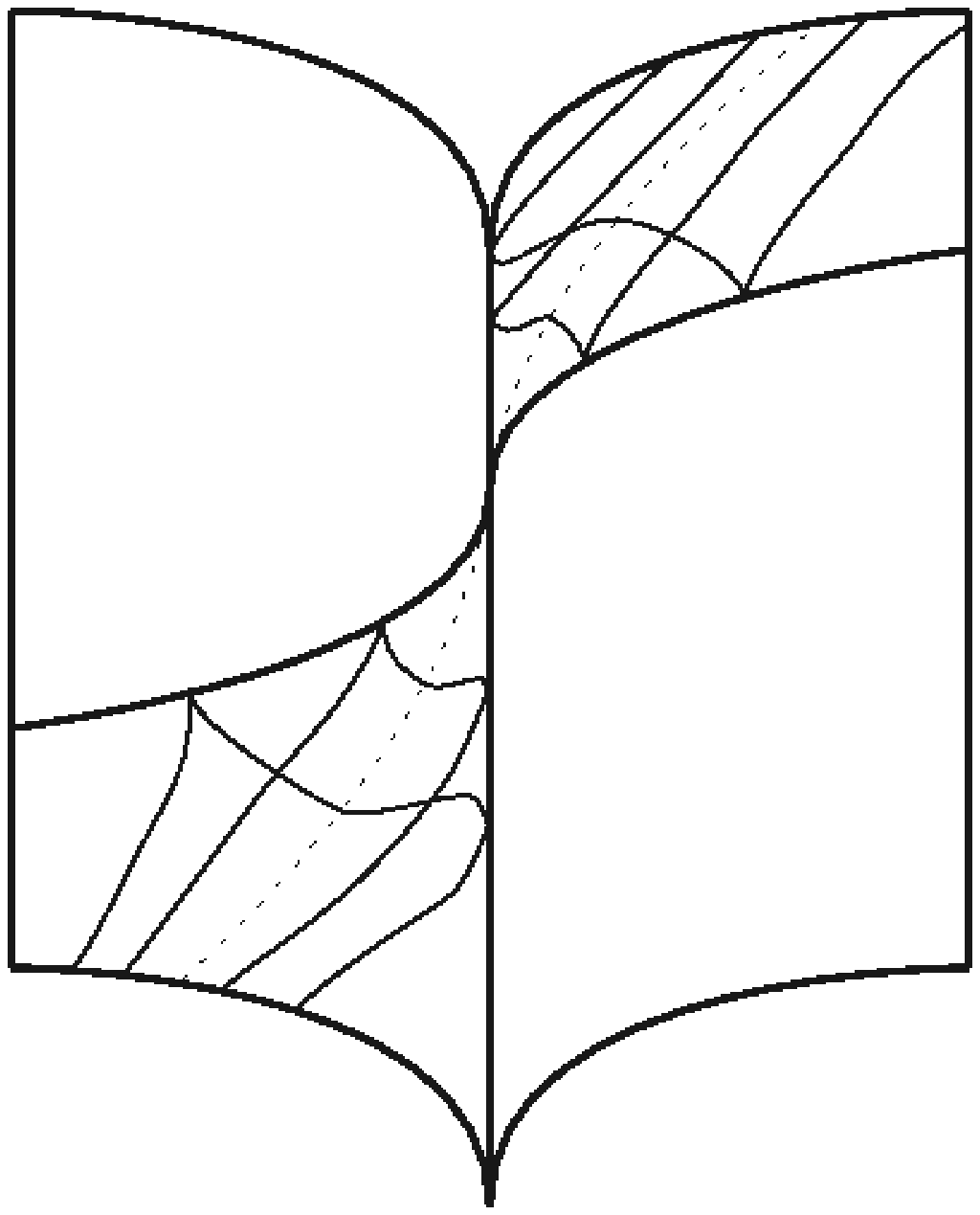}
\hspace{1mm}
\includegraphics[width=.3\linewidth]{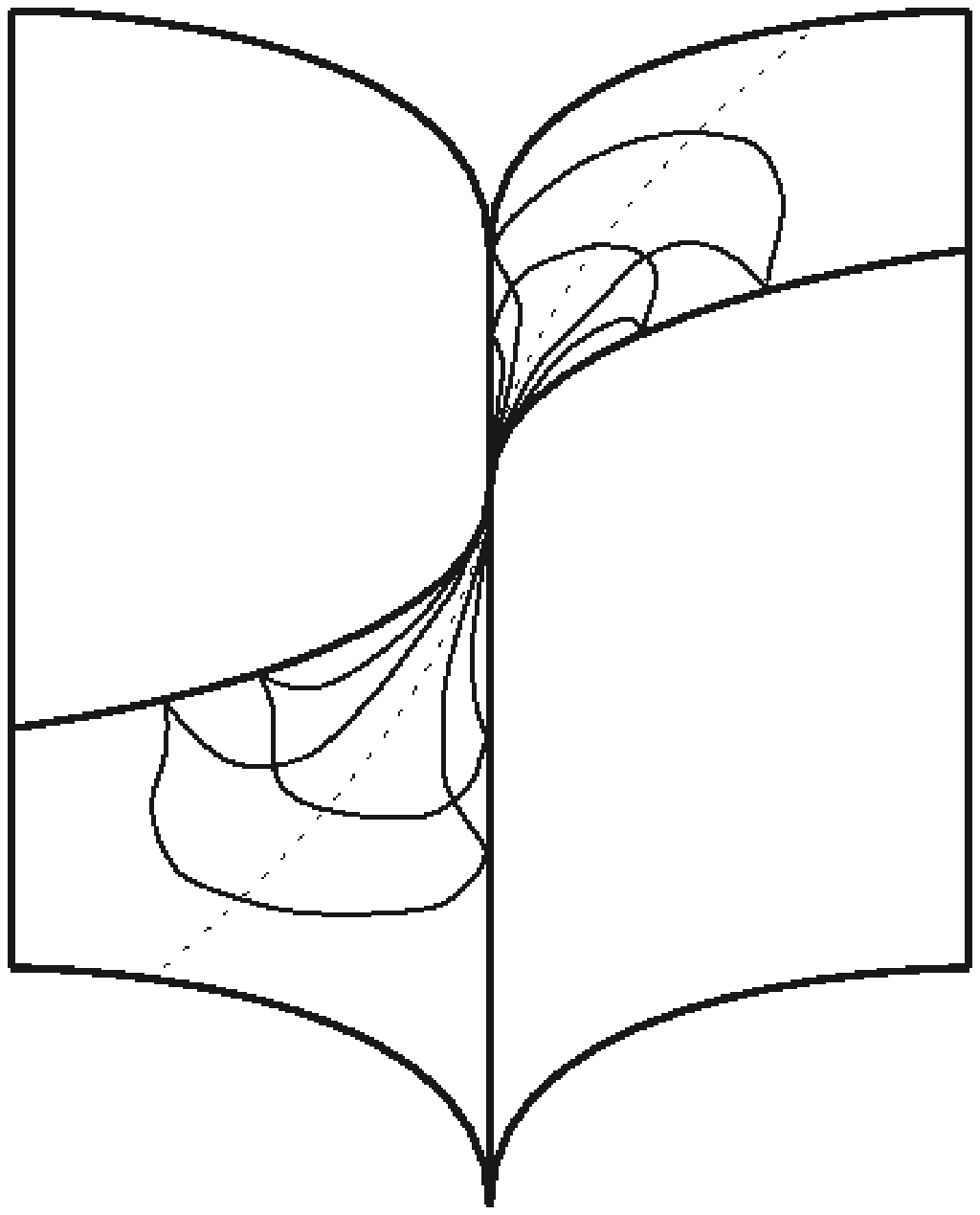}
\caption{Integral curves of
$\omega_{1s}$ and $\omega_{1n}$ on 
images of cuspidal edges.}
\label{fig:ce3}
\end{figure}

\section{Generic foliations}
In 
Propositions \ref{prop:asym} and \ref{prop:char}, 
all the conditions are written in
terms of the $3$-jet of \eqref{eq:west2}.
We can state a genericity result for cuspidal edge.
By \eqref{eq:west2}, we identify the set of jets
of parametrization of cuspidal edges
with $(0,0)$
$$
\begin{array}{l}
\C_k=\Big\{\big(j^ka_1(0),j^kb_2(0),j^kb_3(0),j^kb_4(0,0)\big)
\in J^k(1,1)^3\times J^k(2,1)
\,\Big|\,\\[4mm]
\hspace{0mm}
a_1(0)=a_1'(0)=b_2(0)=b_2'(0)=b_3(0)=0,\ 
b_2'(0)>0,\ b_4(0,0)\ne0\Big\},
\end{array}
$$
for $k\geq3$.
With notation as before, consider
$$(a_{20},a_{30},\ldots,b_{20},b_{30},
\ldots,b_{12},\ldots,b_{03},\ldots)$$
as a coordinate system of $\C_k$ (cf. \eqref{eq:west2}).
Define a subset of $\C_k$ by
$$
\begin{array}{l}
\N
=
\{b_{20}=0,\ b_{30}-a_{20}b_{12}=0\}
\cup
\{b_{20}=0,\ D_{asy}=0\}\\
\hspace{10mm}
\cup
\{b_{20}=0,\ 4b_{12}^3-b_{03}^2b_{30}=0\}
\cup
\{b_{20}=0,\ D_{hmc}=0\}.
\end{array}
$$
Then $\N$ is an algebaric subset of $\C_k$ of
codimension $2$.
Since the singular set of a cuspidal edge 
is a curve, generically it will avoid the set $\N$.
This implies that for generic cuspidal edges
the configuration of $\omega_{lc},\omega_{as},\omega_{ch}$
are those in Propositions
\ref{prop:lc}, 
\ref{prop:asym}, 
\ref{prop:char}.
\appendix
\section{Criteria for $(3,4)$ and $(3,4,5)$-cusp}
\label{sec:cri}
In this section, we state criteria for
$(3,4)$ and $(3,4,5)$-cusp.
Set
$$
c_1(t)=(t^3,t^4,0,\ldots,0),\quad
c_2(t)=(t^3,t^4,t^5,0,\ldots,0),
$$
and a map-germ
$\gamma:(\R,0)\to(\R^n,0)$, where $n\geq2$ (respectively, $n\geq 3$)
is called {\it $(3,4)$-cusp}\/ (respectively, {\it $(3,4,5)$-cusp}\/)
if $\gamma$ is $\A$-equivalent to the map-germ 
$c_1$ (respectively, $c_2$) at $0$.
\begin{proposition}\label{prop:cri}
A map-germ\/
$\gamma:(\R,0)\to(\R^n,0)$, where\/ $n\geq3$ $($respectively, $n\geq 2)$
is\/ $(3,4,5)$-cusp\/ $($respectively, $(3,4)$-cusp\/$)$
if and only if
\begin{itemize}
\item[{\rm (i)}]\label{itm:cusp1} 
$\gamma'(0)=\gamma''(0)=0$,
\item[{\rm (ii)}]\label{itm:cusp2} 
$\gamma^{(3)}(0),\gamma^{(4)}(0)$ and\/ $\gamma^{(5)}(0)$
are linearly independent
$($respectively, linearly dependent, and\/
$\gamma^{(3)}(0)$ and\/ $\gamma^{(4)}(0)$ 
are linearly independent\/$)$
where\/
$(~)^{(i)}=d^i/dt^i$.
\end{itemize}
\end{proposition}
Although this proposition is known \cite{bg,gh},
we give a sketch of proof for the readers who are not
familiar with it.
\begin{proof}
Since
$$
\Phi(\gamma(t))^{(i)}=d\Phi_0(\gamma^{(i)}(0))
\quad (i=3,4,5)
$$
holds for a map $\Phi:(\R^n,0)\to(\R^n,0)$
under the assumption $\gamma'(0)=\gamma''(0)=0$,
it is obvious that the conditions do not depend on
the parameter and the coordinate system on $\R^n$.

To show the 
proposition,
it is enough to show that
$(t^3+a_1t^5,t^4+a_2t^5,a_3t^5,0,\ldots,0)+O(6)$
is $\A$-equivalent to
$(t^3,t^4,a_3t^5,0,\ldots,0)+O(6)$,
where $a_1,a_2,a_3\in\R$.
Considering the parameter change
$
t\mapsto
t-a_2t^2/4-(16 a_1 + 3 a_2^2)t^3/48,
$
it can be proved.
\end{proof}
Using Proposition \ref{prop:cri}, we have the following:
\begin{proposition}
Let\/ $f:(\R^2,0)\to(\R^3,0)$ be a cuspidal edge and\/
$\gamma:(\R,0)\to(\R^2,0)$ an ordinary cusp such that\/
$df_0(\gamma''(0))=0$.
Then\/ $\hat\gamma=f\circ \gamma$ is a\/ $(3,4)$-cusp
at\/ $0$.
\end{proposition}
\begin{proof}
Without loss of generality, one can assume that
$f$ is given by the form \eqref{eq:west2},
and $\gamma(t)=(t^3+a_4t^4+a_5t^5+O(6),t^2)$ $(a_4,a_5\in\R)$, 
because $df_0(\partial v)=0$.
Then $\hat\gamma(t)=(t^3+a_4t^4+a_5t^5,t^4/2,0)+O(6)$.
By Proposition \ref{prop:cri}, we have the conclusion.
\end{proof}


\toukouchange{
\medskip
{\small
\begin{flushright}
\begin{tabular}{l}
Department of Mathematics,\\
Graduate School of Science, \\
Kobe University, \\
Rokko, Nada, Kobe 657-8501, Japan\\
  E-mail: {\tt sajiO\!\!\!amath.kobe-u.ac.jp}
\end{tabular}
\end{flushright}
}}{}
\end{document}